\documentclass[psamsfonts]{amsart}
\usepackage{amsmath,amstext,amssymb,amsopn,amsthm}
\usepackage{amsmath,amssymb,amsthm}
\usepackage[mathscr]{eucal}
\usepackage{bbm}
\usepackage{color}
\usepackage{tikz}
\usepackage{mathtools}
\usetikzlibrary{arrows}
\usepackage[unicode,bookmarks,colorlinks]{hyperref}
\hypersetup{
    linkcolor=brickred,
}

\definecolor{mahogany}{cmyk}{0, 0.77, 0.87, 0}
\definecolor{salmon}{cmyk}{0, 0.53, 0.38, 0}
\definecolor{melon}{cmyk}{0, 0.46, 0.50, 0}
\definecolor{yellowgreen}{cmyk}{0.44, 0, 0.74, 0}
\definecolor{brickred}{cmyk}{0, 0.89, 0.94, 0.28}
\definecolor{OliveGreen}{cmyk}{0.64, 0, 0.95, 0.40}
\definecolor{RawSienna}{cmyk}{0, 0.72, 1.0, 0.45}
\definecolor{ZurichRed}{rgb}{1, 0, 0} 

\numberwithin{equation}{section}

\newtheorem{thm}{Theorem}[section]        
\newtheorem{cor}{Corollary}[section]
\newtheorem{lem}{Lemma}[section]
\newtheorem{prop}{Proposition}[section]
\newtheorem{rmk}{Remark}[section]

\newenvironment{•}{•}{•}
\newcommand{\R}{\mathbb{R}}                  
\newcommand{\Rd}{\R^d}               
\newcommand{\set}[1]{ \left\{#1\right\} }
\newcommand{\mysum}[3]{\sum\limits_{#1=#2}^{#3}}          

\newcommand{\E}[1]{\mathbb{E}\left[#1\right]}

\newcommand{\eid}{\,\,{\buildrel \mathcal{D} \over =}\,\,}   

\newcommand{\abs}[1]{\left|#1\right|}

\newcommand{\tgo}{t\downarrow 0}

\newcommand{\Prob} {\mathbb{P} }
\newcommand{\palp}{p_{t}^{(\alpha)}}

\newcommand{\Hset}[2]{\mathbb{H}_{#1,#2}^{(\alpha)}(t)}
 
\newcommand{\inalp}{ \frac{1}{2}} 
\newcommand{\ld}{\abs{\Omega}}

\newcommand{\pthesis}[1]{\left(#1\right)}
\newcommand{\alpr}{0<\alpha<2}

\newcommand{\bd}{\partial\Omega}

\newcommand{\myH}{\mathbb{H}}

\newcommand{\myP}{\mathcal{P}}
\newcommand{\myd}{\Omega}
\newcommand{\myp}{{\bf X}=\set{X_{s}}_{s\geq 0}}
\newcommand{\mypr}{{\bf X}}
\newcommand{\myet}{\tau_{\myd}^{(\alpha)}}
\newcommand{\myhc}{Q_{\myd}^{(\alpha)}(t)}
\newcommand{\tal}{t^{-\frac{1}{\alpha}}}
\newcommand{\ptal}{t^{\frac{1}{\alpha}}}
\usepackage[margin=1.25in]{geometry}
\begin{document}

\title[  Heat content ]{ On the one  dimensional  spectral Heat content  for stable processes}
\author{Luis Acu\~na Valverde}

\address{Department of Mathematics, Universidad de Costa Rica}
\email{guillemp22@yahoo.com}
\maketitle

\begin{abstract}
This paper provides the second term in  the small time asymptotic expansion of the spectral heat content of  a rotationally invariant $\alpha$--stable process, $0<\alpha \leq 2$, for the bounded interval $(a,b)$. The  small time behavior of the spectral heat content turns out to be linked to  the   distribution of the supremum and infimum processes.
\end{abstract}

\section{introduction}
Let $0<\alpha\leq 2$ and consider $\myp$ a rotationally  invariant  $\alpha$-stable process in $\Rd$ where $d\geq1$  and whose transition densities, denoted along the paper by  $p_t^{(\alpha)}(x,y)$, are known to satisfy the following  properties.
\begin{enumerate}
\item[$(i)$] $p_t^{(\alpha)}(x,y)$ is radial.  Namely, $\palp(x,y)=\palp(\abs{x-y})$.
\item[$(ii)$] Scaling property: $\palp(\abs{x-y})=
t^{-\frac{d}{\alpha}}p_1^{(\alpha)}(t^{-\frac{1}{\alpha}} \abs{x-y})$.
\item[$(iii)$]  $\palp(x,y)$ satisfies the following two sided estimates  for all $\alpr$. There exists a constant $c_{\alpha,d}>0$ such that
\begin{align}\label{tcomp}
c_{\alpha,d}^{-1}\min\set{t^{-d/\alpha},\frac{t}{\abs{x-y}^{d+\alpha}}}\leq
\palp(x-y)\leq c_{\alpha,d}\min\set{t^{-d/\alpha},\frac{t}{\abs{x-y}^{d+\alpha}}},
\end{align}
for all $x,y\in \R^d$ and $t>0$. See \cite{Chen1} for further details. 
\item[$(iv)$] According to \cite[Theorem 2.1]{Blum}, for all $\alpr$, we have
\begin{align}\label{tlim}
\lim_{\tgo}\frac{\palp(\abs{x-y})}{t}=\frac{A_{\alpha,d}}{\abs{x-y}^{d+\alpha}},
\end{align}
for all $x\neq y$, where 
\begin{equation}\label{Adef}
A_{\alpha,d}=\alpha \, 2^{\alpha-1}\, \pi^{-1-\frac{d}{2}}\,\sin\pthesis{\frac{\pi\alpha}{2}}\,\Gamma\pthesis{\frac{d+\alpha}{2}}\,\Gamma\pthesis{\frac{\alpha}{2}}.
\end{equation} 

\end{enumerate} 

Before continuing,  we introduce the following standard notation. $\mathbb{E}^x$ and $\Prob^x$ will denote the expectation 
and  probability of any  process started at $x$, respectively. Also  for  simplicity, 
we will connote $\Prob=\Prob^{0}$, $\mathbb{E}=\mathbb{E}^{0}$  and write 
$Z\eid Y$ for two random variables $Z,Y$ with values in $\Rd$ to mean that they are equal 
in distribution or have the same law.

Let us at this point  establish the following convention which is needed  to provide some references and motivation. When $d=1$, $\Omega$ will stand for an interval $(a,b)$ with finite length $b-a$ denoted by $\ld$. For $d>1$, $\myd$ will be a bounded domain with smooth boundary $\bd$ and volume denoted by $\ld$. In addition, we set
\begin{align}\label{dersa}
\abs{\bd} =
  \begin{cases}
    \#\set{x\in \R: x\in \bd},\quad \text{if } d =1,\\
     \quad \text{surface area of $\myd$,}  \,\,\,\,\,\,\text{ if}\,\,\, d\geq 2.\\
  \end{cases}
\end{align}

Given $\myd\subseteq\Rd$ as above, we consider for $\myp$ the first exit time upon $\myd$. That is,
\begin{align*}
\myet=\inf\set{s \geq 0: X_s \in \myd^c }.
\end{align*} 
The purpose of the paper is to investigate the small time behavior of the following  function, which is called {\it the spectral heat content over $\myd$} and it is given by
\begin{align*}
\myhc=\int_{\myd}dx\,\Prob^x\pthesis{\myet>t},\,\, t>0,
\end{align*}
when $\Omega=(a,b)$ with $\ld=b-a<\infty$. Of course, $\myhc$ makes sense even in the higher dimensional setting when $\Omega$ is taken according to our convention about $\myd$.

It is worth noting that the \textit{spectral heat content} of $\Omega$ takes an alternative form. In fact,
\begin{align}\label{shc}
{Q}_{\Omega}^{(\alpha)}(t)=\int_{\Omega}dx\,\int_{\Omega}dy\,\,p^{\Omega,\alpha}_t(x,y),
\end{align}
where $p^{\Omega,\alpha}_t(x,y)$ is the transition density for the stable process killed upon exiting $\Omega$.  An explicit expression is given by 
\begin{equation}\label{tran.den.dom}
p^{\Omega,\alpha}_t(x,y)=\palp(x,y)\,\,\mathbb{P}\left(\tau_{\Omega}^{(\alpha)}>t \,\,|\,\, X_0=x,\,\,X_t=y\right).
\end{equation}

The name {\it spectral heat content}  given to $Q_{\Omega}^{(\alpha)}(t)$ comes from the fact that $p^{\Omega,\alpha}_t(x,y)$ can be written in terms of the eigenvalues and eigenfunctions of the underlying domain 
 $\Omega$. That is, it is known (see \cite{Davies}) that there exists an orthonormal basis of eigenfunctions $\{\phi_n \}_{n\in \mathbb{N}}$ for $L^2(\Omega)$  with  corresponding eigenvalues $\{\lambda_n \}_{n \in \mathbb{N}}$
satisfying $0 < \lambda_1 < \lambda_2 \leq \lambda_3 \leq . . .$
and  $\lambda_n \rightarrow \infty$ as $n \rightarrow \infty$ such that
\begin{equation*}
p^{\Omega,\alpha}_t(x,y)=\mysum{n}{1}{\infty}e^{-t\lambda_n}\,\phi_{n}(x)\,\phi_n(y).
\end{equation*}
Notice that due to \eqref{shc} and the last equality, we obtain an expression for 
$Q_{\Omega}^{(\alpha)}(t)$ involving both the spectrum $\set{ \lambda_n}_{ n\in \mathbb{N}}$ and  eigenfunctions $\set{\phi_{n}}_{n\in \mathbb{N}}$. That is,
$$Q_{\Omega}^{(\alpha)}(t)=\mysum{n}{1}{\infty}e^{-t\lambda_{n}}\pthesis{\int_{\Omega}dx\,\phi_n(x)}^2.$$

The interest in investigating the spectral heat content is to  obtain   geometry features of the underlying set $\myd$ such as volume, surface area, mean curvature, ect. The spectral heat content has been deeply studied so far for the Brownian motion  ($\alpha=2$) and we refer the interested reader to 
\cite{van2, van4, vanden3, vandenDryKap} for further details concerning asymptotics of the spectral heat content corresponding to the Brownian motion for  different kind of domains.  As for the cases $0<\alpha<2$, the author of this paper in \cite{Acu0} has provided estimates of $\myhc$ involving the volume of $\myd$, its surface  area and the fractional $\alpha$--perimeter (see \eqref{per} below). In fact, the following conjecture about 
the asymptotic expansion of  $\myhc$ is provided in  \cite{Acu0}.
\bigskip

{\it {\bf Conjecture:}} Let $\Omega$ be an interval  with finite length ($d=1$) or a bounded domain with smooth boundary  ($d\geq 2$). Then,  
\begin{enumerate} 
\item[$(i)$] for $1<\alpha<2$, there exists $C_{d,\alpha}>0$ such that
\begin{align*}
Q_{\Omega}^{(\alpha)}(t)&=\abs{\Omega}-C_{d,\alpha}\,\abs{\bd} \, t^{\frac{1}{\alpha}} +\mathcal{O}(t),\,\, \tgo,
\end{align*}

\item[$(ii)$] for $\alpha=1$, there exists $C_{d}>0$ such that
\begin{align*}
\hspace{7mm}Q_{\Omega}^{(1)}(t)&=\abs{\Omega}-C_{d}\,
\abs{\bd}  \, t\,\ln\pthesis{\frac{1}{t}} +\mathcal{O}(t), \,\, \tgo,
\end{align*}

\item[$(iii)$] for $0<\alpha<1$, there exists $\gamma_{d,\alpha}>0$ such that
\begin{align*}
Q_{\Omega}^{(\alpha)}(t)&=\abs{\Omega}-\gamma_{d,\alpha}\myP_{\alpha}(\Omega) \, t +o(t),\,\, \tgo,
\end{align*}
where 
\begin{align}\label{per}
\myP_{\alpha}(\Omega)=\int_{\myd}\int_{\myd^c}\frac{dxdy}
{|x-y|^{d+\alpha}}
\end{align}
is known as {\it the fractional $\alpha$--perimeter}. (We refer the reader to \cite{Frank,FracP} for further details.)
\end{enumerate}

The study of $\myhc$ when $\myd$ is an interval with finite length is motivated by the aforementioned  conjecture and our main result provides the exact constants $C_{1,\alpha}$, $C_1$ and $\gamma_{1,\alpha}$.  However, we do not give any estimates concerning the remainders.

The preceding conjecture was proved bearing in mind the small time behavior of the  function 
\begin{align}\label{hc}
\myH_{\Omega,\Omega^c}^{(\alpha)}(t)=\int_{\Omega}dx\int_{\Omega^c}dy \,\palp(x,y)=\ld-\int_{\Omega}dx\int_{\Omega}dy \,\palp(x,y),
\end{align}
since for small values of $t$, it is expected that the heat kernels
$\palp(x,y)$ and $p^{\Omega,\alpha}_t(x,y)$ behave alike.

The function $\myH_{\Omega,\Omega^c}^{(\alpha)}(t)$ will play an important role in the proof of the main result of this paper, so that we need the following two facts about $\myH_{\Omega,\Omega^c}^{(\alpha)}(t)$ whose proofs can be found in \cite{Acu0}. First, by appealing to \eqref{tran.den.dom}, we see that  $p^{\Omega,\alpha}_t(x,y)\leq \palp(x,y)$. Thus, we easily obtain
\begin{align}\label{ineq1}
\myH_{\Omega,\Omega^c}^{(\alpha)}(t)\leq |\Omega|-\myhc.
\end{align}
Secondly, according to Theorem 1.1 in \cite{Acu0}, we have for
 $\Omega=(a,b)$ with $\ld=b-a<\infty$ the following.
\begin{enumerate}
\item[$(i)$] For $1<\alpha\leq 2$,
\begin{align}\label{case1}
\lim\limits_{\tgo}\frac{\Hset{\Omega}{\Omega^c}}{\ptal}=\frac{2}{\pi}\,\Gamma\pthesis{1-\frac{1}{\alpha}}.
\end{align}
\item[$(ii)$] For $\alpha=1$, the following  equality holds.
\begin{align}\label{case2}
\lim\limits_{\tgo}\frac{\myH_{\Omega, \Omega^c}^{(1)}(t)}{t\ln\pthesis{\frac{1}{t}}}=\frac{2}{\pi}.
\end{align}
\item[$(iii)$] Let $0<\alpha<1$. Then,
\begin{align}\label{case3}
\lim\limits_{\tgo}\frac{\myH_{\Omega, \Omega^c}^{(\alpha)}(t)}{t}=A_{\alpha,1}\myP_{\alpha}(\Omega),
\end{align} 
with $A_{\alpha,1}$ and $\myP_{\alpha}(\Omega)$ as given in \eqref{Adef} and \eqref{per}, respectively.
\end{enumerate}

Before proceeding with the statement of our main result, we introduce  two stochastic processes associated with the one dimensional $\alpha$-stable process ${\bf X}$. For $t\geq0$ and ${\bf X}$ started from $0$, we set
\begin{align}\label{infsupproc}
\overline{X}_{t}&=\sup\set{X_{s}:0\leq s\leq t},\\
\nonumber
\underline{X}_{\,t}&=\,\inf\set{X_{s}:0\leq s\leq t}.
\end{align}
The stochastic processes $\overline{\bf X}=\set{\overline{X}_{s}}_{s\geq 0}$  and
$\underline{\bf X}=\set{\underline{X}_{s}}_{s\geq 0}$ are called the supremum and infimum process, respectively.  The distribution of the  process $\overline{\bf X}$ has been widely investigated by many authors (see for example \cite{Chaumont}, \cite{Chaumont1} and \cite{Darling}) because of its connections with fluctuation and excursion theory  and major interest in
stochastic modelling, such as queuing and risk theories.

The main result of this paper is the following.

\begin{thm}\label{mymth}
Consider $\Omega=(a,b)$ with  $\ld=b-a<\infty$. 
\begin{enumerate}
\item[$a)$] Let $1<\alpha\leq 2$. Then, we have
$$\lim\limits_{\tgo}\frac{\abs{\myd}-Q^{(\alpha)}_{\myd}(t)}{\ptal}=2\,\, \E{
\overline{X}_1}.$$
\item[$b)$] For $\alpha=1$, we obtain
$$\lim\limits_{\tgo}\frac{\ld-Q^{(1)}_{\Omega}(t)}{t\ln\pthesis{\frac{1}{t}}}=\frac{2}{\pi}.$$
\item[$c)$] Let $0<\alpha<1$. Then,
$$\lim\limits_{\tgo}\frac{\ld-Q^{(\alpha)}_{\Omega}(t)}{t}=A_{\alpha,1}\myP_{\alpha}(\Omega),$$
with $A_{\alpha,1}$ and $\myP_{\alpha}(\Omega)$ as given in \eqref{Adef} and \eqref{per}, respectively.
\end{enumerate}
\end{thm}
We remark that the main idea in the proof of Theorem \ref{mymth} consists in expressing $\myhc$ as  quantities involving   the distribution of the supremum and infimum processes  related to the $\alpha$-stable process ${\bf X}$. 
The idea of employing the distribution of the supremum comes from the paper of van den Berg and le Gall \cite{van2}, where they use in the case of the Brownian motion ($\alpha=2$) the supremum  process, the  unique a.s. time $\tau_{t}$ satisfying $X_{\tau_t}=\overline{X}_{t}$ and the independence among coordinates to prove  
for smooth domains $\myd \subset \R^d$, $d\geq 2$ that
\begin{equation}\label{VandenberHeatContBM}
Q_{\myd}^{(2)}(t)=|\myd|-\frac{2}{\sqrt{\pi}}
\abs{\bd}t^{1/2}+\left(2^{-1}(d-1)\int_{\partial \myd}\mathcal{M}(s)ds\right) t +\mathcal{O}(t^{3 /2}),  
\end{equation}
as $t\downarrow 0$.  Here, $\mathcal{M}(s)$ represents the mean curvature at the point $s\in \partial \myd$.  

We point out that when $\alpha=2$ and using that for an interval
$\abs{\bd}=2$ (see \eqref{dersa}), we obtain by Theorem \ref{mymth}, the one dimensional analogue to \eqref{VandenberHeatContBM}. That is,
$$\lim\limits_{\tgo}\frac{\abs{\myd}-Q^{(2)}_{\myd}(t)}{t^{\frac{1}{2}}}=\frac{2}{\sqrt{\pi}}\abs{\bd},$$
since, according to \eqref{Bmcal} below, we have $\E{\overline{X}_1}=\frac{2}{\sqrt{\pi}}$.

The rest of the paper is organized as follows. In \S \ref{sec:T1}, we establish properties concerning the distribution of the supremum and infimum processes.  In \S \ref{sec:T2}, we  rewrite $\myhc$ in terms of integrals involving the distribution of $\bf \overline{X}$ and $\bf \underline{X}.$ In \S \ref{sec:T4}, we provide the proof of Theorem \ref{mymth}, where the proof of a  crucial proposition is postponed until  \S \ref{sec:T5}.

\section{properties of the supremum and infimum processes}\label{sec:T1} 

As we have mentioned in the introduction, the proof of our main result depends on the properties of the distribution of the supremum and infimum processes. For this reason, this section is dedicated to establishing properties about $\overline{X}$ and $\underline{X}$ only when $0<\alpha<2$ since for $\alpha=2$ these properties are already well known.

To begin with, consider $A\subseteq \R$ a Borel-set and the first hitting time for $A$ defined by
$$T_A=\inf\set{s>0:X_s\in A}.$$
We recall that $x$ is regular for $A$ if
$\Prob^{x}\pthesis{T_A=0}=1.$
That is, the process starting from $x$ meets $A$ for arbitrarily small strictly positive times with probability one.

Next, due to the fact that an $\alpha$-stable process $\bf X$ satisfies by symmetry that $\Prob\pthesis{X_t>0}=\frac{1}{2}$, we have 
$$\int_{0}^
{1}dt\,\frac{1}{t}\Prob\pthesis{X_t>0}=\infty.$$ 
The divergence of the last integral guarantees by Theorem 6.5 in \cite{Kyp}  that  $0$ is regular for $(0,\infty)$ which in turn  implies that $\overline{X}_t> 0$ a.s. due to the fact that the process $\bf X$ starting from $0$ will hit  $(0,\infty)$ for arbitrarily small strictly positive times with probability one. Since $ \overline{X}_{t}\eid
-\underline{X}_{\,t}$, we also obtain $\underline{X}_{\,t}< 0$ a.s.

The next proposition will be very useful in the following. Roughly speaking, it says  that the distribution of the random variables  $\overline{X}_t$ and $X_t$ are comparable. 

\begin{prop}\label{prop1}
For every $u>0$ and $t>0$, we have 
\begin{align*}
\Prob\pthesis{u\leq X_t}\leq
\Prob\pthesis{u\leq \overline{X}_t} \leq
2\,\Prob\pthesis{u\leq X_t}.
\end{align*}
\end{prop}
\begin{proof}
We start by noticing that 
$\Prob\pthesis{u \leq X_t  }=
\Prob\pthesis{u\leq \overline{X}_t,u\leq X_t  }\leq \Prob\pthesis{u\leq \overline{X}_t}$. It remains to show that
$$\Prob\pthesis{u\leq \overline{X}_t} \leq
2\,\Prob\pthesis{u\leq X_t}.$$
Denote $T_{u} = T_{(u,\infty)}= \inf\set{s \geq 0 : X_s \geq u}$, the first passage time over $u$ and observe that
$$\set{T_u\leq t}=\set{u\leq \overline{X}_t}.$$ Therefore, by  a conditioning argument, we arrive at
\begin{align}\label{pro1}
\Prob\pthesis{u \leq X_t  }&=
\Prob\pthesis{u\leq \overline{X}_t,u\leq X_t  }\\ \nonumber
&=\Prob\pthesis{u\leq X_t \big{|}u\leq \overline{X}_t}
\Prob\pthesis{u\leq \overline{X}_t}\\ \nonumber
&=\Prob\pthesis{u\leq X_t \big{|}T_u\leq t}
\Prob\pthesis{u\leq \overline{X}_t}\\ \nonumber
&=\Prob\pthesis{u-X_{T_u}\leq X_{(t-T_u)+T_u}-X_{T_u} \big{|}T_u\leq t}
\Prob\pthesis{u\leq \overline{X}_t}.
\end{align}
Now, by   the definition of $T_{u}$, we have $u-X_{T_u}\leq 0$. Thus,
\begin{align}\label{pro2}
\Prob\pthesis{u-X_{T_u}\leq X_{(t-T_u)+T_u}-X_{T_u} \big{|}T_u\leq t}\geq \Prob\pthesis{0\leq X_{(t-T_u)+T_u}-X_{T_u} \big{|}T_u\leq t}.
\end{align}
Finally, by the strong Markov property and the symmetric property of the $\alpha$-stable process, we obtain
\begin{align}\label{pro3}
\Prob\pthesis{0\leq X_{(t-T_u)+T_u}-X_{T_u} \big{|}T_u\leq t}=\frac{1}{2}.
\end{align} 
Hence, the desired result is proved by combining \eqref{pro1}, \eqref{pro2} and \eqref{pro3}.
\end{proof}
 
 The following lemma  provides estimates about the tail behavior of 
 the $\alpha$-stable process $\bf X$ and its proof  is omitted because it consists in a elementary integration problem.
\begin{lem}\label{lem2}
Let $\psi(z)=\min\set{1,\abs{z}^{-(1+\alpha)}}$ with $0<\alpha<2$. Then, for all $u>0$, we have
\begin{align*}
\int_{u}^{\infty}dz\,\psi(z)=
\mathbbm{1}_{(0,1)}(u)\pthesis{1+\frac{1}{\alpha}-u}+\mathbbm{1}_{[1,\infty)}(u)\pthesis{\frac{u^{-\alpha}}{\alpha}}.
\end{align*}
In particular, due to \eqref{tcomp}, it follows  for $t<u^{\alpha}$ that
\begin{align}\label{Probest}
c_{\alpha,1}^{-1}\frac{t}{u^{\alpha}\alpha}\leq \Prob\pthesis{u\leq X_t}=\Prob\pthesis{\frac{u}{\ptal}\leq X_1}\leq c_{\alpha,1}\frac{t}{u^{\alpha}\alpha}.
\end{align}
\end{lem}
As a consequence of combining  Lemma \ref{lem2} together with Proposition \ref{prop1}, we conclude the following result.
\begin{cor}\label{mycor}
\hspace{20mm}
\begin{enumerate}
\item[$(i)$] For all $t>0$,
\begin{align*}
\lim\limits_{u\rightarrow \infty}\Prob\pthesis{u\leq \overline{X}_t}=\lim\limits_{u\rightarrow \infty}\Prob\pthesis{u\leq X_t}=0.
\end{align*}
\item[$(ii)$] $$\int_{0}^{\infty}du\, \Prob(u\leq \overline{ X}_1)<\infty$$ if and only if 
$$\int_{0}^{\infty}du\, \Prob(u\leq  X_1)<\infty$$ if and only if $1<\alpha<2$.
\\
\hspace{20mm}
\item[$(iii)$] For all $t>0$,
\begin{align*}
\Prob(\overline{X}_t=\infty)=\Prob(\underline{X}_{\,t}=-\infty)=0.
\end{align*}
\end{enumerate}
\end{cor}

\begin{proof}
We only need to prove part $(iii)$. It is a basic fact in probability theory that $$\Prob(\overline{X}_t=\infty)=\lim\limits_{n\rightarrow\infty}\Prob\pthesis{n\leq \overline{X}_t}.$$ Hence, by Proposition \ref{prop1} and Lemma \ref{lem2}, we arrive by part $(i)$ at
 $$\Prob(\overline{X}_t=\infty)\leq 2\, \lim\limits_{n\rightarrow\infty}\Prob\pthesis{n\leq X_t}=0.$$
 
 On the other hand, using that $\overline{X}_{t}\eid-\underline{X}_{\,t}$, we have
 $\Prob(\overline{X}_t=\infty)=\Prob(\underline{X}_{\,t}=-\infty)$
and this completes the proof. 

\end{proof}

\section{spectral heat content in terms of  the supremum and infimum processes}\label{sec:T2}
We start this section by expressing  the spectral heat content $\myhc$ in a more convenient form. Henceforth,  $\myd=(a,b)$ with $\ld=b-a<\infty$ and $0<\alpha\leq 2$.

\begin{lem}\label{lem1} For $x \in \myd$ and $t>0$, we set
$\myd_t(x)=\pthesis{(x-b)\tal,(x-a)\tal}$. Then,
\begin{align*}
\Prob^x\pthesis{\myet>t}=
\Prob\pthesis{\tau^{(\alpha)}_{\myd_t(x)}>1}.
\end{align*}
\end{lem}

\begin{proof} The proof is based on the scaling property $X_{t\ell}\eid t^{\frac{1}{\alpha}}X_{\ell}$ and  the translation property of the rotationally invariant $\alpha$-stable process.
\begin{align*}
\Prob^x\pthesis{\myet>t}&=\Prob^x\pthesis{a\leq X_s \leq b, \forall \,0\leq s\leq t}\\  
&= \Prob\pthesis{a\leq x-\ptal X_{\ell} \leq b, \forall \,0\leq \ell \leq 1} \\  
&= \Prob\pthesis{ (x-b)\tal \leq  X_{\ell} \leq (x-a)\tal, \forall \,0\leq \ell \leq 1}\\ 
&=\Prob\pthesis{\tau^{(\alpha)}_{\myd_t(x)}>1}.
\end{align*}
\end{proof}

The upcoming proposition allows us to decompose $\myhc$ as a sum involving  the distribution of the supremum and infimum processes.

\begin{prop} \label{ml}
For every  $t>0$, we have
\begin{align*} 
\myhc=\ld -2\, \ptal \int_{0}^{\ld \tal}du\,
\Prob\pthesis{u\leq \overline{X}_{1}} +\ptal r^{(\alpha)}(t),
\end{align*}
where
$$r^{(\alpha)}(t)=\int_{0}^{\ld \tal}du\,
\Prob\pthesis{u\leq \overline{X}_{1},\,\,\underline{X}_{\,1}\leq u-\ld\, \tal}.$$
\end{prop}
\begin{proof}
Because of Lemma \ref{lem1}, we know that
$$\myhc=\int_{a}^{b}dx\,\Prob\pthesis{\tau^{(\alpha)}_{\myd_t(x)}>1}.$$
Now, the change of variable $x=\ptal u+a$ yields 
$$\myd_t(x)=\pthesis{(x-b)\tal,(x-a)\tal}=(u-\ld \tal,u)$$ and
\begin{align}\label{kd}
\myhc&=\ptal\int_{0}^{\ld \tal}du\,
\Prob\pthesis{\tau^{(\alpha)}_{(u-\ld \tal,u)}>1}\\\nonumber
&=\ld -\,\ptal\int_{0}^{\ld \tal}du\,
\Prob\pthesis{\tau^{(\alpha)}_{(u-\ld \tal,u)}\leq 1}.
\end{align}
The key step towards the desired decomposition is based on the fact that
\begin{flalign*}
&\Prob\pthesis{\tau^{(\alpha)}_{(u-\ld \tal,u)}\leq 1}=\Prob\pthesis{\set{u\leq \overline{X}_{1} }\cup\set{\underline{X}_{\,1}\leq u-\ld\, \tal}}=\\
\nonumber
&\Prob\pthesis{u\leq \overline{X}_{1} }
+\Prob\pthesis{\underline{X}_{\,1}\leq u-\ld\, \tal}-\Prob\pthesis{u\leq \overline{X}_{1},\,\, \underline{X}_{\,1}\leq u-\ld\, \tal}.
\end{flalign*}
Next,  integrate the last expression from $0$ to
$\ld \tal$ and replace it into \eqref{kd}. To finish the proof, it suffices to observe that the change of variable $v=u-\ld \tal$ and the fact $ \overline{X}_{1}\eid
-\underline{X}_{\,1}$ yield
\begin{align*}
\int_{0}^{\ld \tal}du\,\Prob\pthesis{\underline{X}_{\,1}\leq u-\ld\, \tal}&=\int_{-\ld \tal}^{0}dv\,\Prob\pthesis{\underline{X}_{\,1}\leq v}\\ \nonumber
&=\int_{-\ld \tal}^{0}dv\,\Prob\pthesis{-\overline{X}_{\,1}\leq v}=
\int_{0}^{\ld \tal}du\,\Prob\pthesis{u\leq \overline{X}_{\,1}}.
\end{align*}
\end{proof}

\section{proof of theorem \ref{mymth}}\label{sec:T4}
Let us set for $t>0$,
\begin{align}\label{Ldef}
L^{(\alpha)}(t)=\int_{0}^{\ld \tal }du\,\Prob\pthesis{u\leq \overline{X}_1},
\end{align}
and  $$L^{(\alpha)}(0)=\int_{0}^{\infty }du\,\Prob\pthesis{u\leq \overline{X}_1}.$$
According to Lemma 2 in \cite{Chaumont1} (see also \cite{Doney, Chaumont, Darling}), it is known that there exists a strictly positive and continuous density $f_{\alpha}(x)$ defined over $(0, \infty)$ such  that
\begin{align}\label{dsup}
\Prob\pthesis{\overline{ X}_1 \in A}=\int_{A}dx \,f_{\alpha}(x),
\end{align}
for any Borel set $A\subseteq [0,\infty)$.
\begin{rmk}
We point out that $L^{(\alpha)}(0)<\infty$ only for $1<\alpha\leq 2$. To see this, we observe that Proposition \ref{prop1} ensures the finiteness of $L^{(\alpha)}(0)$ for $1<\alpha<2$. Regarding the case $\alpha=2$, $\mypr$ represents a Brownian Motion at twice velocity. It is  well known that for every $u>0$, the following equality holds.
\begin{align} \label{mdbm}
\Prob\pthesis{ \underline{X}_t\leq -u }=\Prob\pthesis{u\leq \overline{X}_t }=  2\, \Prob\pthesis{u\leq X_t}.
\end{align}
 Therefore, we have by interchanging the order of integration that
\begin{align}\label{Bmcal}
L^{(2)}(0)&=2\,\int_{0}^{\infty}du\, \Prob(u\leq X_1)=2\,\int_{0}^{\infty}du\, \int_{u}^{\infty}dz\,p_1^{(2)}(z)\\\nonumber &=2 \int_{0}^{\infty}dz\,z\,p_1^{(2)}(z)=2\,\int_{0}^{\infty}dz \,\frac{z}{\sqrt{4\pi}}\,\exp\pthesis{-\frac{z^2}{4}}=\frac{2}{\sqrt{\pi}}.
\end{align}

It is a basic fact in probability (see \cite{Ross})  that the existence of the density $f_{\alpha}(x)$ and the finiteness of 
$L^{(\alpha)}(0)$
imply that $L^{(\alpha)}(0)=\E{\overline{X}_1}$. 

We also mention that Proposition \ref{prop1} yields 
$\E{X_1,X_1>0}\leq \E{\,\overline{X}_1}\leq 2\,\E{X_1,X_1>0}.$ In addition, it is proved in \cite[page 14]{Acu0}, by subordination of the Gaussian kernel that
$$\E{X_1,X_1>0}=\frac{1}{\pi}\Gamma\pthesis{1-\frac{1}{\alpha}},$$
whenever $1<\alpha\leq 2$.
\end{rmk}

The following proposition provides the small time behavior of $L^{(\alpha)}(t)$ as $\tgo$ with a remainder term as long as $1<\alpha\leq2.$
\begin{prop} \label{p1}
Consider $1<\alpha\leq 2$. Then,
$$L^{(\alpha)}(t)=\E{\overline{X}_1}+\mathcal{O}(t\cdot \mathbbm{1}_{\set{2}}(\alpha)+t^{1-\frac{1}{\alpha}}\cdot \mathbbm{1}_{(1,2)}(\alpha)), \mbox{as}\,\,\,  \tgo.$$ 
In particular,$$\lim\limits_{\tgo}L^{(\alpha)}(t)=\E{\overline{X}_1}.$$

\end{prop}
\begin{proof} It is clear by \eqref{Ldef} and remark \ref{dsup} that
$$L^{(\alpha)}(t)=\E{\overline{X}_1}-\int_{\ld \tal}^{\infty}du\,\Prob(u\leq \overline{X}_1).$$
Now, by Proposition \ref{prop1}, we have
$$\int_{\ld \tal}^{\infty}du\,\Prob(u\leq \overline{X}_1)\leq 2\,\int_{\ld \tal}^{\infty}du\,\Prob(u\leq X_1).$$
Let us denote the integral term at the right hand side of the previous inequality by $R^{(\alpha)}(t)$. To complete the proof of the proposition, it suffices to estimate how fast $R^{(\alpha)}(t)$ tends to $0$ as $\tgo$.
We  proceed to consider cases.

\hspace{20mm}
\item[{\bf Case $\alpha=2$}:] By interchanging the order of integration, we obtain
\begin{align*}
R^{(2)}(t)&=\int_{\ld t^{-\frac{1}{2}}}^{\infty}du \int_{u}^{\infty}dz\,p_1^{(2)}(z)
=\int_{\ld t^{-\frac{1}{2}}}^{\infty}dz\,p_1^{(2)}(z)
\int_{\ld t^{-\frac{1}{2}}}^{z}du\\
&=\E{X_1, \ld t^{-\inalp}< X_1 }-\ld t^{-\inalp}\,
\Prob \pthesis{\ld t^{-\inalp}< X_1}.
\end{align*}

Let us denote 
\begin{align}\label{rm}
j(t)=\E{X_1, \ld t^{-\inalp}< X_1 }
\end{align}
and observe that
\begin{align*}
j(t)\geq \ld t^{-\inalp}\,\Prob\pthesis{\ld t^{-\inalp}< X_1}.
\end{align*}
Thus, the remainder function $R^{(2)}(t)$ satisfies $0\leq R^{(2)}(t)\leq 2\,j(t).$ 
It follows from \eqref{rm} that  
\begin{align*}
j(t)&=(4\pi)^{-1/2}
\int_{ \ld t^{-1/2}}^{\infty}dz\,z\exp\left(-\frac{z^2}{4}\right)=\pi^{-1/2}\exp\left(-\frac{\ld^2}{4t}\right).
\end{align*}

Next, by applying the elementary inequality 
\begin{align*}
\exp(-x)\leq x^{-1},\,\, x>0,
\end{align*}
we conclude that $j(t)\leq 4\pi^{-1/2}\ld^{-2}t.$  
\bigskip

\item[{\bf Case $1<\alpha< 2$:}]
By inequality \eqref{Probest}, we have for $\ld \tal >1$ that 
$$0\leq R^{(\alpha)}(t)\leq \frac{c_{\alpha,1}}{\alpha}\int_{\ld \tal}^{\infty}du\,u^{-\alpha}=\pthesis{\frac{c_{\alpha,1}\ld^{1-\alpha}}{\alpha(\alpha-1)}}t^{1-\frac{1}{\alpha}}.$$ 
\end{proof}

The last ingredient in the proof of part $(a)$ of our main result is to show that $r^{(\alpha)}(t)\rightarrow 0$ as $\tgo$.

\begin{prop} \label{p2}
Consider $1<\alpha\leq 2$ and
$$r^{(\alpha)}(t)=\int_{0}^{\ld \tal}du\,
\Prob\pthesis{u\leq \overline{X}_{1},\underline{X}_{\,1}\leq u-\ld\, \tal}.$$
Then,
\begin{align*}
\lim\limits_{\tgo}r^{(\alpha)}(t)=0.
\end{align*}
\end{prop}
\begin{proof}
Define $$G_{\alpha}(t,u)=\mathbbm{1}_{\pthesis{0, \ld \tal}}(u)\cdot\Prob\pthesis{u\leq \overline{X}_{1},\,\underline{X}_{\,1}\leq u-\ld\, t^{-\frac{1}{\alpha}}},$$ which satisfies that
$$G_{\alpha}(t,u)\leq \Prob\pthesis{u\leq \overline{X}_{1}},$$ where according to  Corollary \ref{mycor} (for $1<\alpha<2$) and equality \eqref{Bmcal} (for $\alpha=2$), we have
$ \Prob\pthesis{u\leq \overline{X}_{1}}\in L^{1}((0,\infty)).$ Furthermore, $$G_{\alpha}(t,u)\leq \Prob\pthesis{\underline{X}_{\,1}\leq u-\ld\, t^{-\frac{1}{\alpha}}},$$
which in turn implies by appealing once more to Corollary \ref{mycor}  and equality \eqref{mdbm} that 
$$\lim\limits_{\tgo}G_{\alpha}(t,u)\leq \Prob\pthesis{\underline{X}_{\,1}=-\infty}=0.$$
Thus, the result follows from an application of the Dominated Convergence Theorem.
\end{proof}

{\bf Proof of part $a)$ of Theorem \ref{mymth}:} 
\begin{proof}

Let $1<\alpha\leq 2$. By Proposition \ref{ml}, we obtain 
\begin{align*}
\lim\limits_{\tgo}\frac{\abs{\myd}-\myhc}{\ptal}=\lim\limits_{\tgo}2\,L^{(\alpha)}(t)-r^{(\alpha)}(t).
\end{align*}
Thus, the desired result follows from Proposition \ref{p1} and  Proposition \ref{p2}.
\end{proof} 

As for the  rest of the proof of our main result,  the following proposition concerning how fast $L^{(\alpha)}(t)$ diverges to $\infty$ as $\tgo$ when $0<\alpha\leq 1$ will be a crucial ingredient. It is remarkable that the behavior  of $L^{(\alpha)}(t)$ for small values of $t$ has been deduced from the function $\myH_{\Omega,\Omega^c}^{(\alpha)}(t)$ given in \eqref{hc} and the limits described in \eqref{case2} and \eqref{case3}.

\begin{prop}\label{propfinal}
\hspace{20mm}
\begin{enumerate}
\item[$(i)$] Let $\alpha=1$. Then, $$\lim\limits_{\tgo}\frac{L^{(1)}(t)}{\ln\pthesis{\frac{1}{t}}}=\frac{1}{\pi}.$$
\item[$(ii)$] For $0<\alpha<1$, we have
$$\lim\limits_{\tgo}\frac{\ptal\,L^{(\alpha)}(t)}{t}=\frac{1}{2}A_{\alpha, 1}\myP_{\alpha}(\Omega).$$
Here,  $A_{\alpha,1}$ and $\myP_{\alpha}(\Omega)$ are the constants defined in \eqref{Adef} and \eqref{per}, respectively.
\end{enumerate}
\end{prop}

In order to make this proof as clear as possible, we first proceed to complete the proof of the Theorem \ref{mymth} and we postpone the proof of Proposition \ref{propfinal} until the last section.
\\

{\bf Proof of  part $(b)$ and $(c)$ of Theorem \ref{mymth}}
\begin{proof}
To begin with, we notice that by the inequality \eqref{ineq1} and Proposition \ref{ml}, we arrive at the following inequality.
 \begin{align}\label{fineq}
\myH_{\Omega,\Omega^c}^{(\alpha)}(t)\leq |\Omega|-\myhc
\leq  2\,\ptal\,L^{(\alpha)}(t), 
\end{align}
with $\myhc$ and $L^{(\alpha)}(t)$ as given in  \eqref{shc} and \eqref{Ldef}, respectively. Therefore, the   desired result is obtained by applying the Squeeze Theorem to the inequality \eqref{fineq}  since by \eqref{case2}, \eqref{case3} and Proposition \ref{propfinal}, we have

\begin{align*}
\lim\limits_{\tgo}\frac{\myH_{\Omega, \Omega^c}^{(1)}(t)}{t\ln\pthesis{\frac{1}{t}}}&=\lim\limits_{\tgo}\frac{2\,t\,L^{(1)}(t)}{t\,\ln\pthesis{\frac{1}{t}}}=\frac{2}{\pi},\\
\lim\limits_{\tgo}\frac{\myH_{\Omega, \Omega^c}^{(\alpha)}(t)}{t}&=\lim\limits_{\tgo}\frac{2\,\ptal\,L^{(\alpha)}(t)}{t}=A_{\alpha, 1}\myP_{\alpha}(\Omega), \mbox{\,\,for\,\,} 0<\alpha<1.
\end{align*}
\end{proof}

\section{proof of proposition \ref{propfinal}}\label{sec:T5}

{\bf Case $\alpha=1$:} For this case, $\bf X$ is called Cauchy process.
It was proved by Darling in \cite{Darling} that the density of the supremum corresponding to the Cauchy process (see 
\eqref{dsup}) is given by
\begin{align*}
f_1(x)=\frac{1}{\pi(1+x^2)}F\pthesis{\frac{1}{x}},\,\,\, x>0,
\end{align*}
where 
$$F\pthesis{z}
=\exp\pthesis{\frac{1}{\pi}\int_{0}^{\infty}\ln\pthesis{z+y}\cdot
\frac{dy}{1+y^2}}.$$

\begin{rmk}\label{imprmk}
An elementary calculation shows that for $x>0$
$$\int_{0}^{x}\ln(y)\cdot\frac{dy}{
1+y^2}=-\int_{\frac{1}{x}}^{\infty}\ln(y)\cdot\frac{dy}{
1+y^2},
$$
which in turn implies by taking $x=1$ that
$\int_{0}^{\infty}\ln(y)\cdot\frac{dy}{
1+y^2}=0$ and $F(0)=1$. 
\end{rmk}

We have then
\begin{align*}
L^{(1)}(t)=\int_{0}^{\ld t^{-1} }du\,\Prob\pthesis{u\leq \overline{X}_1}=\int_{0}^{\ld t^{-1} }du\int_{u}^{\infty}dx\,f_1(x).
\end{align*}

{\bf Proof of part $(i)$ of Proposition \ref{propfinal}:}
\begin{proof}
We start by observing that due to the  Fundamental Theorem of Calculus and the facts that $\Prob\pthesis{\ld t^{-1}\leq \overline{X}_1}\rightarrow 0$ and  $L^{(1)}(t)\rightarrow \infty$ as $\tgo$ (see Corollary \ref{Probest}), we obtain by   applying  L'H\^{o}pital's rule twice that
\begin{align*}
\lim\limits_{\tgo}\frac{L^{(1)}(t)}{\ln\pthesis{\frac{1}{t}}}&=\lim\limits_{\tgo}\frac{\ld\,\Prob\pthesis{\ld t^{-1}\leq \overline{X}_1}}{t}=\lim\limits_{\tgo}\ld^{2}\,t^{-2}\,f_1(\ld\,t^{-1})\\
&=\lim\limits_{\tgo}\frac{\ld^2\,F\pthesis{t\ld^{-1}}}{\pi(t^2+\ld^2)}=\frac{F(0)}{\pi}=\frac{1}{\pi}.
\end{align*}
In the last equality above, we have used  according to 
remark \ref{imprmk} that $F(0)=1$.
\end{proof}

{\bf Proof of part $(ii)$ of Proposition \ref{propfinal}:}
\begin{proof}
Along the proof, we assume $0<\alpha<1$. For  $\Omega=(a,b)$ with $\ld =b-a<\infty$, the fractional $\alpha$-perimeter $\mathcal{P}_{\alpha}\pthesis{\Omega}$  defined in \eqref{per}
is equal by  simple integrations to
\begin{align}\label{Perexp}
\myP_{\alpha}\pthesis{\Omega}&=
\int_{\Omega}\int_{\Omega^c}\frac{dx\,\,dy}{|x-y|^{1+\alpha}}= 
\int_{a}^{b}dx\pthesis{\int_{-\infty}^{a}\frac{dy}{(x-y)^{1+\alpha}}+\int_{b}^{\infty}\frac{dy}{(y-x)^{1+\alpha}}}\\ \nonumber
&=\frac{1}{\alpha}\int_{a}^{b}dx\pthesis{(x-a)^{-\alpha}+(b-x)^{-\alpha}}=
\frac{2}{\alpha(1-\alpha)}\ld^{1-\alpha}.
\end{align}

Now, for $0<\alpha<1$, we know that $L^{(\alpha)}(t)$ and $t^{1-\frac{1}{\alpha}}$ both tend to $\infty$ as $\tgo$ by Corollary \ref{Probest}. Thus, we have by   applying  L'H\^{o}pital's rule that
\begin{align}\label{lim1}
\lim\limits_{\tgo}
\frac{t^{\frac{1}{\alpha}}\,L^{(\alpha)}(t)}{t}=
\lim\limits_{\tgo}
\frac{L^{(\alpha)}(t)}{t^{1-\frac{1}{\alpha}}}=
\frac{\ld}{1-\alpha}\lim\limits_{\tgo}
\frac{\Prob\pthesis{\ld\,\tal\leq \overline{X}_{1}}}{t}.
\end{align}

Next, according to  Proposition 4,  in \cite[page 221]{Ber}, it is known that
$$\lim\limits_{\tgo}
\frac{\Prob\pthesis{\ld\,\tal\leq \overline{X}_{1}}}{\Prob\pthesis{\ld\,
\tal \leq X_1}}=1.$$ 
Hence, we deduce that
\begin{align}\label{lim0}
\lim\limits_{\tgo}
\frac{\Prob\pthesis{\ld\,\tal\leq \overline{X}_{1}}}{t}=
\lim\limits_{\tgo}
\frac{\Prob\pthesis{\ld\,\tal\leq X_{1}}}{t},
\end{align}
and by applying L'H\^{o}pital's rule one more time, we arrive at
\begin{align}\label{lim2}
\lim\limits_{\tgo}
\frac{\Prob\pthesis{\ld\,\tal\leq X_{1}}}{t}=\frac{\ld}{\alpha}\lim\limits_{\tgo}t^{-1-\frac{1}{\alpha}}p_1^{(\alpha)}\pthesis{\ld \tal}.
\end{align}
We point out that the transition densities of the one dimensional $\alpha$--stable process ${\bf X}$ are known to satisfy  that $$\palp(x)=\palp(\abs{x})=t^{-\frac{1}{\alpha}} \,p_1^{(\alpha)}(t^{-\frac{1}{\alpha}} \abs{x}),
$$
which in turn implies by \eqref{tlim} that 
\begin{align}\label{lim3}
\lim\limits_{\tgo}t^{-1-\frac{1}{\alpha}}p_1^{(\alpha)}\pthesis{\ld \tal}=\lim\limits_{\tgo}\frac{\palp(\ld)}{t}=
\frac{A_{\alpha,1}}{\ld^{1+\alpha}}.
\end{align}
Then, by combining $\eqref{lim1}$, $\eqref{lim0}$, $\eqref{lim2}$
and $\eqref{lim3}$, we obtain that
$$\lim\limits_{\tgo}
\frac{t^{\frac{1}{\alpha}}\,L^{(\alpha)}(t)}{t}=\frac{A_{\alpha, 1}\ld^{1-\alpha}}{\alpha(1-\alpha)}=\frac{A_{\alpha, 1}\myP_{\alpha}(\myd)}{2},$$
where the last equality is a consequence of \eqref{Perexp} and this completes the proof.
\end{proof}
{\bf Acknowledgements}: I would like to thank the referee whose comments and corrections have improved the quality  of the paper. This  investigation has been supported by Universidad de Costa Rica, project 1563.

\end{document}